\documentclass[amstyle,reqno]{amsart}
\usepackage{amsmath,color}
\numberwithin{equation}{section}

\input epsf

\usepackage{amssymb}
\usepackage{amsfonts}
\usepackage{times}
\usepackage{latexsym}
\theoremstyle{plain}
\newtheorem{theorem}{Theorem}[section]
\newtheorem{CO}[theorem]{Corollary}
\newtheorem{LE}[theorem]{Lemma}
\newtheorem{PR}[theorem]{Proposition}

\newtheorem*{theorem*}{Theorem}

\theoremstyle{definition}

\newtheorem{RE}[theorem]{Remark}

\newcounter{claim_nb}[theorem] 
\theoremstyle{remark}
\newtheorem{claim}[claim_nb]{Claim}

\newcommand{\cal}[1]{\mathcal #1}

\newcommand{\zC}{\cal C}

\newcommand{\zF}{\cal F}

\def\VFS{transversal}
\newcounter{enumi_counter}
\newenvironment{enumi}
{\begin{list}
	{(\roman{enumi_counter})}
	{\usecounter{enumi_counter}
	 \setlength{\topsep}{0.0in}}}
{\end{list}}

\newenvironment{cproof}
{\begin{proof}
 [Proof of Claim:]
 \vspace{-3.2\parsep}}
{ \end{proof}}

\begin{document}

\setcounter{page}{1}

\title{Packing directed circuits exactly}

\author{Bertrand Guenin}
\author{Robin Thomas}

\thanks{
Research partially supported by NSF under 
Grant No. DMS 96-32032 and
Grant No. DMS-9970514.}
\thanks{{\em Classification:} 
05C20, 
90C47} 

\address{
Bertrand Guenin \newline
Department of Combinatorics and Optimization \newline
Faculty of Mathematics \newline
University of Waterloo \newline
Waterloo, ON N2L 3G1, Canada}

\address{
Robin Thomas \newline
School of Mathematics \newline
Georgia Institute of Technology \newline
Atlanta, Georgia  30332, USA}

\date{January 2001. Revised 1 December 2010}

\begin{abstract}
We give an ``excluded minor" and a ``structural"
characterization of  digraphs $D$ that
have the property that for every subdigraph $H$ of $D$, the maximum
number of disjoint circuits in $H$ is equal to the minimum cardinality
of a set $T\subseteq V(H)$ such that $H\backslash T$ is acyclic.
\end{abstract}


\maketitle
%
%
\section{Introduction}
Graphs and digraphs in this paper may have loops and multiple edges.
 Paths and circuits have no
``repeated" vertices, and in digraphs they are directed.
A {\em \VFS} in a digraph $D$ is a set of
vertices $T$ which intersects every circuit, i.e. $D \backslash T$ 
is acyclic. A {\em packing of circuits} (or {\em packing} for short)
is a collection  of pairwise (vertex-)disjoint circuits.
The cardinality of a minimum \VFS\  is denoted by $\tau(D)$ and the
cardinality of a maximum packing  is denoted by $\nu(D)$.
Clearly $\nu(D)\le\tau(D)$, and our objective is to study when equality
holds. We will show in Section~\ref{sec-planar} that this is the case
for every strongly planar digraph. (A digraph is {\em strongly planar}
if it has a planar drawing such that for every vertex $v$, the edges
with head $v$ form an interval in the cyclic ordering of edges incident with
$v$.)
However, in general there is
probably no nice characterization of digraphs for which equality holds,
and so instead we characterize digraphs such that equality holds for
{\em every subdigraph}. Thus
we say that a digraph $D$ {\em packs} if $\tau(D')=\nu(D')$ for every 
subdigraph $D'$ of $D$. 

We will give two characterizations: one in terms
of excluded minors, and the other will give a structural description
of digraphs that pack.
We say that an edge $e$ of a digraph $D$ with head $v$ and tail $u$
is {\em special} if either $e$ is the only edge of $D$ with head $v$,
or it is the only edge of $D$ with tail $u$, or both. We say that
a digraph $D$ is a {\em minor} of a digraph $D'$ if $D$ can be obtained
from a subdigraph of $D'$ by repeatedly contracting special edges.
It is easy to see that if a digraph packs, then so do all its minors.
Thus digraphs that pack can be characterized by a list of minor-minimal
digraphs that do not pack. By an {\em odd double circuit} we mean
the digraph obtained from an undirected circuit of odd length at least three
by replacing each edge by a pair of directed edges, one in each direction. 
The digraph $F_7$ is defined in Figure~\ref{f7-fig}.
\begin{figure}[!htb]
\epsfysize=1.8in
\epsfbox{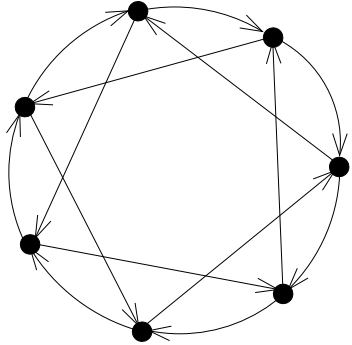}
\caption{The digraph $F_7$.} \label{f7-fig}
\end{figure}
\noindent
The following is our excluded minor characterization.
\begin{theorem} \label{main-th}
A digraph packs if and only if it has no minor isomorphic to an odd double circuit or $F_7$.
\end{theorem}

If $D$ is an odd double circuit with $k$ vertices then 
$\tau(D)=\lceil k/2 \rceil > \nu(D)=\lfloor k/2 \rfloor$. 
Moreover, $\tau(F_7)=3 > \nu(F_7)=2$. 
Thus odd double circuits and $F_7$ do not pack and the content of 
Theorem~\ref{main-th} is to prove the converse.

The structural characterization can be stated directly in terms of
digraphs, but it is more convenient to rephrase it in terms of bipartite
graphs, and therefore we postpone its statement until
Section~\ref{sec-braces}. 
Roughly, the characterization states that a digraph packs if and only if
it can be obtained from strongly planar digraphs by means of certain 
composition operations.  

Our main tool in the proof is a characterization of bipartite graphs that have
a Pfaffian orientation, found independently by McCuaig \cite{McCuaig:97}
and by Robertson, Seymour and the second author \cite{Robertson:99}.
We present the characterization in Section~\ref{sec-braces}.
The rest of the paper is organized as follows. In Section~\ref{related}
we mention three related results.
Section~\ref{sec-2conn} reduces the problem to strongly 2-connected digraphs.
It is shown in Section~\ref{sec-planar} that strongly planar digraphs pack.
Sections~\ref{sec-2sum}-\ref{sec-4sum} show that the property that digraphs pack
is preserved under the composition operations of the characterization theorem,
thus completing the proof of Theorem~\ref{main-th}. 
Finally, Section~\ref{sec-conclusion} offers some closing remarks.
%
%
\section{Related Results}\label{related}  
In this section we review three related results.  The first is a classical
theorem of Lucchesi and Younger, of which we only state a corollary
\cite{Lucchesi:78}(Theorem B).
\begin{theorem}\label{thm2.1}
Let $D$ be a planar digraph and $\zF$ be the family of its directed circuits.
Then for any set of weights $w: E(D) \rightarrow Z_+$ we have,
\begin{equation} \label{eq-LY}
\begin{aligned}
\mbox{ } & \quad \min\{
\sum_{e \in E(D)} w(e) x_e:
\sum_{e \in C} x_e \geq 1, \forall C \in \zF, x \in \{0,1\}^{E(D)}\} \\
= & \quad \max\{
\sum_{C:C \in \zF} y_C:
\sum_{C:e \in C \in \zF} y_C \leq w(e), \forall e \in E(D), y \in Z_+^{\zF}\}.
\end{aligned}
\end{equation}
\end{theorem}
Thus, in particular, in a planar digraph the maximum cardinality of a
collection of edge-disjoint circuits is equal to the minimum cardinality of
a set of edges whose deletion makes the graph acyclic.  This relation does
not hold for all digraphs, but there is an upper bound on $\tau (D)$
as a function of $\nu (D)$.  (A simple construction --- splitting each
vertex into a ``source" and a ``sink," also used in the proof of
Corollary~\ref{planar-pr} --- shows that the same function serves as
an upper bound for both the edge-disjoint as well as vertex-disjoint
version of the problem.  Note, however, that this construction does not
preserve planarity, but it preserves strong planarity.)  
McCuaig \cite{McCuaig:97} characterized
all digraphs $D$ with $\nu (D)\le 1$; the following follows immediately
from his characterization (but there does not seem to be a direct proof).

\begin{theorem}\label{thm2.2}
For every digraph $D$, if $\nu (D)\le 1$, then $\tau(D)\le 3$.
\end{theorem}

In general, Reed, Robertson, Seymour and the second author \cite{ReeRobSeyTho}
proved the following.

\begin{theorem}\label{thm2.3}
There is a function $f$ such that for every digraph $D$
$$\tau (D)\le f(\nu (D)).$$
\end{theorem}

The function $f$ from the proof of Theorem~\ref{thm2.3}, albeit explicit, grows
rather fast.  The best lower bound of $f(k)\ge\Omega(k\log k)$ was obtained
by Noga Alon (unpublished).
%
Finally, the undirected analogue of the problem we study is quite easy.
It becomes much harder if we only require that the equality
$\nu=\tau$ hold for all {\em induced} subgraphs. This problem remains open.
However, Ding and Zang~\cite{DinZan} managed to solve the closely 
related problem of characterizing graphs for which it is required that a weighted 
version of the relation $\nu=\tau$ holds. 
They gave a characterization by means of excluded induced subgraphs, 
and also gave a structural description of those graphs.
We omit the precise statement.
%
%
\section{Strong 2-connectivity} \label{sec-2conn}

Let $D$ be a digraph and $\zC$ a packing of circuits.
We will say that $\zC$ {\em uses a vertex $v$} if there exists a circuit 
$C$ in $\zC$ with $v \in V(C)$. Consider a digraph $D$ that packs.
Then some minimum \VFS\  includes $v$ if and only if 
$\tau(D \backslash v)=\tau(D)-1$.
As $D$ packs, $\nu(D \backslash v)=\tau(D \backslash v)=\tau(D)-1=\nu(D)-1$.
But $\nu (D\backslash v)=\nu (D)-1$ if and only if
every maximum packing uses $v$.
Thus we have shown the following.
\begin{RE} \label{easy-re}
Let $D$ be a digraph that packs.
There exists a minimum \VFS\  of $D$ containing $v$ 
if and only if every maximum packing of $D$ uses $v$.
\end{RE}
A digraph is {\em strongly connected} if for every pair of vertices
$u$ and $v$ there is a path from $u$ to $v$. 
A digraph $D$ is {\em strongly $k$-connected} if for every 
$T \subseteq V(D)$, where $|T| \leq k-1$, 
the digraph $D \backslash T$ is strongly connected.
If $D$ is not strongly connected, then $V(D)$ can be partitioned
into non-empty sets $X_1,X_2$ such that no edge has tail in $X_1$
and head in $X_2$. Let $D_1 := D \backslash X_2$ and $D_2 := D \backslash X_1$.
Then $D$ is said to be a {\em 0-sum} of $D_1$ and $D_2$. 
Since every circuit of $D$ is a circuit of precisely one of $D_1$
or $D_2$, the following is straightforward.
\begin{PR} \label{0sum-pr}
Let $D$ be the $0$-sum of $D_1$ and $D_2$. 
Then $D_1$ and $D_2$ pack if and only 
$D$ packs.
\end{PR}

Suppose $D$ is strongly connected, but not strongly 2-connected;  
thus there is a vertex $v$ such that 
$D \backslash v$ is not strongly connected.
Then there is a partition of $V(D)-\{v\}$ into non-empty sets $X_1,X_2$ 
such that
all edges with endpoints in both $X_1$ and $X_2$ have tail in $X_1$ and head in $X_2$.
Let $F$ be the set of all these edges.
For $i=1,2$ let $D_i$ be the digraph obtained from $D$ by deleting all edges
with both endpoints in $X_{3-i} \cup \{v\}$ and identifying all vertices of $X_{3-i} \cup \{v\}$
into a vertex called $v$. Thus edges of $F$ belong to both $D_1$ and $D_2$; in $D_1$ they
have head $v$ and in $D_2$ they have tail $v$.
We say that $D$ is a {\em 1-sum} of $D_1$ and $D_2$.

Let $D$ be a digraph.  We denote by $D+uv$ the digraph obtained from
$D$ by adding the vertices $u,v$ (if they are not vertices of $D$) and
an edge with tail $u$ and head $v$.  Let
us stress that we add the edge even if $D$ already has one or more edges
with tail $u$ and head $v$.
We use $D+\{u_1v_1,u_2v_2,\dots,
u_kv_k\}$ to denote $D+u_1v_1+u_2v_2+\cdots + u_kv_k$.

\begin{PR} \label{1sum-pr}
Let a strongly connected digraph $D$ be the $1$-sum of $D_1$ and $D_2$.
Then $D_1$ and $D_2$ pack if and only if 
$D$ packs.
\end{PR}

\begin{proof}
Since $D$ is strongly connected, the digraphs $D_1$ and $D_2$ are minors of $D$. 
So if $D$ packs, so do $D_1$ and $D_2$.  Conversely, assume that $D_1$ and $D_2$ pack.
Since every subdigraph of $D$ is either a subdigraph of $D_1$ or $D_2$, or a 
0-sum or 1-sum of subdigraphs of $D_1$ and $D_2$, it suffices to show 
that $\tau(D)=\nu(D)$.
Let $v,X_1,X_2,$ and $F$ be as in the definition of 1-sum.
For $i=1,2$ let $D'_i := D_i \backslash F$ and let $\zC_i$ be a maximum packing of $D'_i$.
Suppose that, for $i=1,2$, every maximum packing of $D'_i$ uses the vertex $v$.
It follows from Remark~\ref{easy-re} that there is a minimum \VFS\  $T_i$ of 
$D'_i$ using $v$. Let $\zC$ be obtained from the union of $\zC_1,\zC_2$ 
by removing the circuit of $\zC_1$ using $v$. Then $\zC$ is a packing of $D$
and $T:=T_1 \cup T_2$ is a \VFS\  of $D$. Moreover, both have cardinality 
$\tau(D'_1)+\tau(D'_2)-1$, i.e. $\tau(D)=\nu(D)$. 
Thus we can assume one of $\zC_i$ ($i=1,2$), say $\zC_1$, does not use the vertex $v$.

For $i=1,2$, let $F_i$ be the set of edges $f$ of $F$ such that $\nu(D'_i+f)=\nu(D'_i)$.
Consider first the case where $F_1 \cup F_2=F$.
Suppose for a contradiction $\nu(D'_i+F_i)>\nu(D'_i)$ 
and let $\zF$ be a corresponding packing. 
Clearly $\zF$ uses an edge of $F_i$. 
Moreover as all edges $F$ of $D_i$ share the endpoint $v$, 
$\zF$ uses exactly one edge $f$ of $F_i$. Hence $\nu(D'_i+f)>\nu(D'_i)$, a contradiction.
Since (for $i=1,2$) $D'_i+F_i$ packs it has a \VFS\  $T_i$ of cardinality $\tau(D'_i)$.
As $F_1 \cup F_2=F$ this implies that $T_1 \cup T_2$ is a \VFS\  of $D$.
Recall that $\zC_1$ does not use $v$; thus $\zC_1 \cup \zC_2$ is a packing of $D$
and $|T_1 \cup T_2|=\tau(D'_1)+\tau(D'_2)=|\zC_1 \cup \zC_2|$, i.e. $\tau(D)=\nu(D)$.

Thus we may assume there exists $f \in F-F_1-F_2$.
Let $\zC'_i$ ($i=1,2$) be a maximum packing of $D'_i+f$.
Each $\zC'_i$ contains a circuit $C_i$ using $f$.
Define $\zC$ to be the collection of all circuits of $\zC_1,\zC_2$ distinct from
$C_1$ and $C_2$ as well as the circuit $(C_1 \cup C_2) \backslash f$ of $D$.
Let $T_i$ ($i=1,2$) be a minimum \VFS\  of $D'_i$.
Then $T:=T_1 \cup T_2 \cup \{v\}$ is a \VFS\  of $D$ and $\zC$ a packing of $D$.
Moreover, $|T|=\tau(D'_1)+\tau(D'_2)+1=|\zC|$, i.e. $\tau(D)=\nu(D)$,
as desired.
\end{proof}
%
%
\section{Strong planarity} \label{sec-planar}
Let us recall that a
 digraph is {\em strongly planar} if it has a planar drawing such that for all
vertices $v$, the edges with head $v$ form an interval in the cyclic ordering 
of edges incident with $v$ determined by the drawing.
\begin{CO} \label{planar-pr}
Every strongly planar digraph packs.
\end{CO}
\begin{proof}
Let $D$ be a strongly planar digraph with vertex set $V$ and edge set $E$.
We will show that $D$ packs. Since subdigraphs of strongly planar digraphs
are strongly planar it suffices to show $\tau(D)=\nu(D)$.
Associate to every vertex $v$ a new vertex $v'$ and let $V'$ be the set of 
all vertices $v'$.
Associate with every edge $e\in E(D)$ with tail $u$ and head $v$ a new edge
$e'$ with tail $u'$ and head $v$. 
We define a digraph $H$ as follows: 
the vertex-set of $H$ is $V\cup V'$, and the edge-set of $H$ consists of
all the edges $e'$ for $e\in E(D)$ and all the edges of the form
$vv'$, where $v\in V(D)$.
Define weights $w: E(H) \rightarrow Z_+$ as follows:
$w(e')=|E(H)|$ for all $e \in E(D)$ and $w(vv')=1$ for all $v \in V(H)$.
It is easy to see that the drawing associated to the strongly planar 
digraph $D$ can be modified to induce a planar drawing of $H$. 
Now equation (\ref{eq-LY}) states $\tau(D)=\nu(D)$, as desired.
\end{proof}
%
%
\section{Braces} \label{sec-braces}
It will be convenient to reformulate our packing problem about 
digraphs to one about bipartite graphs.  Let $G$ be a bipartite graph with 
bipartition $(A,B)$, and let $M$ be a perfect matching in $G$.
We denote by $D(G,M)$ the digraph obtained from $G$ by directing every edge of 
$G$ from $A$ to $B$, and contracting every edge of $M$.  
When $G'$ is a subgraph of $G$ and $M\cap E(G')$ is a perfect matching of
$G'$ we will abbreviate $D(G',M\cap E(G'))$ by $D(G',M)$.  
It is clear that
every digraph is isomorphic to $D(G,M)$ for some bipartite graph $G$
and some perfect matching $M$.  Moreover, the following is straightforward.  
\begin{RE} \label{splanar-re}
Let $G$ be a bipartite graph and let $M$ be a perfect matching in $G$.
If $G$ is planar then $D(G,M)$ is strongly planar.
\end{RE}
A graph $G$ is $k$-extendable, where $k$ is an integer, if every
matching in $G$ of size at most $k$ can be extended to a perfect matching.
A 2-extendable bipartite graph is called a {\em brace}.
The following straightforward relation between $k$-extendability and strong 
$k$-connectivity is very important.
\begin{PR} \label{connectivity}
Let $G$ be a connected bipartite graph, let $M$ be a perfect matching in $G$,
and let $k \geq 1$ be an integer. Then $G$ is $k$-extendable 
if and only if $D(G,M)$ is strongly $k$-connected.
\end{PR}
Let $G$ be a bipartite graph and $M$ a perfect matching in $G$ such that 
$D(G,M)$ is isomorphic to $F_7$.
This defines $G$ uniquely up to isomorphism, and the graph so defined
is called the {\em Heawood graph}.

Let $G$ be a bipartite graph, and let $e$ be an edge of $G$ with ends
$u,v$.  Consider a new graph obtained from $G$ by replacing $e$ by a path
with an even number of vertices and ends $u,v$  
and otherwise disjoint from $G$.
Let $G'$ be obtained from 
$G$ by repeating this operation, possibly for different edges of $G$.
We say that $G'$ is an {\em even subdivision} of $G$.  The graph $G'$ is
clearly bipartite.  Now let $G,H$ be bipartite graphs.  We say that
$G$ {\em contains} $H$ if $G$ has a subgraph $L$ such that 
$G\backslash V(L)$ has a perfect matching, and $L$ is isomorphic to an
even subdivision of $H$.  

A circuit $C$ in a bipartite graph $G$ is {\em central} 
if $G\backslash V(C)$ has a perfect matching.
Let $G_0$ be a bipartite graph, let $C$ be a central circuit of 
$G_0$ of length $4$,
and let $G_1,G_2$ be subgraphs of $G_0$ such that
$G_1 \cup G_2=G_0, G_1 \cap G_2=C$, and
$V(G_1)-V(G_2) \neq \emptyset \neq V(G_2)-V(G_1)$.
Let $G$ be obtained from $G_0$ by deleting all the edges of $C$.
In this case we say that $G$ is the {\em 4-sum} of $G_1$ or $G_2$ along $C$.
This is a slight departure from the definition in \cite{Robertson:99}, but
the class of simple graphs obtainable according to our definition is
the same, because we allow parallel edges.

Let $G_0$ be a bipartite graph, let $C$ be a central circuit of 
$G_0$ of length $4$,
and let $G_1,G_2,G_3$ be three subgraphs of $G_0$ such that:
$G_1 \cup G_2 \cup G_3 = G_0$ and for distinct integers $i,j \in \{1,2,3\}$
$G_i \cap G_j = C$ and $V(G_i) - V(G_j) \neq \emptyset$.
Let $G$ be obtained from $G_0$ by deleting all the edges of $C$.
In these circumstances we say that $G$ is a {\em trisum} of $G_1,G_2,G_3$ {\em along $C$}.
We will need the following result.

\begin{theorem}\label{pretool}
Let $G$ be a brace, and let $M$ be a perfect matching in $G$.  Then the 
following conditions are equivalent.
\begin{itemize}
\item[\rm (i)] $G$ does not contain $K_{3,3}$,
\item[\rm(ii)] either $G$ is isomorphic to the Heawood graph, or
$G$ can be obtained from planar braces by repeatedly applying the trisum 
operation,
\item[\rm(iii)] either $G$ is isomorphic to the Heawood graph, or
$G$ can be obtained from planar braces by repeatedly applying the 4-sum 
operation,
\item[\rm(iv)] $D(G,M)$ has no minor isomorphic to an odd double circuit.
\end{itemize}
\end{theorem}

\begin{proof} The equivalence of (i), (ii) and (iii) is the main result
of \cite{McCuaig:97} and \cite{Robertson:99}.  Condition (iv) is equivalent
to the other three by results of Little~\cite{Little:75} and 
Seymour and Thomassen~\cite{SeyTho}.  See also~\cite{McCuaig:97}.
\end{proof}

We will need the following small variation of Theorem~\ref{pretool}.

\begin{theorem}\label{maintool}
Let $G$ be a brace, and let $M$ be a perfect matching in $G$.  Then the
following conditions are equivalent.
\begin{itemize}
\item[\rm (i)] $G$ does not contain $K_{3,3}$ or the Heawood graph,
\item[\rm(ii)] $G$ can be obtained from planar braces by repeatedly 
applying the trisum operation,
\item[\rm(iii)] $G$ can be obtained from planar braces by repeatedly 
applying the 4-sum operation,
\item[\rm(iv)] $D(G,M)$ has no minor isomorphic to an odd double circuit
or $F_7$.
\end{itemize}
\end{theorem}

\begin{proof} This follows from Theorem~\ref{pretool} and the fact 
\cite[Theorem~6.7]{Robertson:99} that if $G$ contains the Heawood graph and
is not isomorphic to it, then it contains $K_{3,3}$.
\end{proof}

%
%

We deduce the following information about a minimal counterexample to 
Theorem~\ref{main-th}.

\begin{PR}\label{pr4.8}
Let $G$ be a bipartite graph and $M$ a perfect matching in $G$ such that the
digraph $D:=D(G,M)$ has no minor isomorphic to an odd double circuit or
$F_7$, and every digraph $D'$ with $|V(D')|+|E(D')|<|V(D)|+|E(D)|$ and no
minor isomorphic to an odd double circuit or $F_7$ packs.  If $\nu (D)<\tau(D)$,
then $G$ is a brace and there exist braces $G_1,G_2,G_3$ such that
$G$ is a trisum of $G_1,G_2,G_3$ along a circuit $C$, and each of
$G_1,G_2,G_3$ can be obtained from planar braces by repeatedly applying
the trisum operation.
\end{PR}

\begin{proof}
It follows from Propositions~\ref{0sum-pr} and \ref{1sum-pr} that
$D$ is strongly 2-connected.  Thus $G$ is a brace by 
Proposition~\ref{connectivity}.  By Corollary~\ref{planar-pr} the digraph
$D$ is not strongly planar, and hence $G$ is not planar by 
Remark~\ref{splanar-re}.  By Theorem~\ref{maintool} the graph $G$ is 
obtained from planar braces by 
repeatedly applying the trisum operation.
Since $G$ itself is not 
planar, there is at least one trisum operation involved in the
construction of $G$, and hence $G_1,G_2,G_3$ and $C$ exist, as desired.
\end{proof}

In the next three sections we will prove the following result.

\begin{PR}\label{mainprop} 
Let $G$, $M$, and  $D$ be as in 
Proposition~\ref{pr4.8}.  Then $\nu (D)=\tau(D)$.
\end{PR}

\begin{proof}[Proof of Theorem~\ref{main-th} 
{\rm (assuming Proposition~\ref{mainprop})}]  
We have already established the ``only if" part.  To prove the ``if"
part let $D$ be a digraph with no minor isomorphic to an odd double
circuit or $F_7$ such that every digraph $D'$ with $|V(D')|+|E(D')|<
|V(D)|+|E(D)|$ and no minor isomorphic to an odd double circuit or
$F_7$ packs.  By Proposition~\ref{mainprop} we have
that $\nu(D)=\tau(D)$, and hence $D$ packs, as desired.
\end{proof}

We now deduce the structural characterization of digraphs that pack.

\begin{CO}\label{coro} A digraph packs if and only if it can be
obtained from strongly\/ $2$-connected digraphs that pack by means of\/ $0$-
and\/ $1$-sums.  A strongly $2$-connected digraph packs if and only if
it is isomorphic to $D(G,M)$ for some brace $G$ and some perfect matching
$M$ in $G$, where $G$ is obtained from planar braces by repeatedly 
applying the trisum operation.
\end{CO}

\begin{proof} The first statement follows from Propositions~\ref{0sum-pr}
and \ref{1sum-pr}.  
For the second statement let $D$ be a strongly 2-connected digraph.
Assume first that $D$ packs, and let $G$ be a bipartite graph and $M$ a
perfect matching such that $D$ is isomorphic to $D(G,M)$.  By 
Proposition~\ref{connectivity} the graph $G$ is a brace. By 
Theorem~\ref{main-th} the digraph $D$ has no minor isomorphic to an odd
double circuit or $F_7$, and so by Theorem~\ref{maintool} $G$ is as desired.
The converse implication follows along the same lines.
\end{proof}

As we alluded to in the Introduction, the second part of 
Corollary~\ref{coro} can be stated purely in terms of ``sums" of digraphs.
However, three kinds of sum are needed (see~\cite{Robertson:99}), as opposed to just one.  
Therefore the formulation we chose is clearer, despite the disadvantage that
it involves the transition from a digraph to a bipartite graph.

Finally, we deduce a corollary about packing $M$-alternating
circuits in bipartite graphs.  Let $G$ be a bipartite graph, and let $M$ be
a perfect matching in $G$.  A circuit $C$ in $G$ is $M$-{\em alternating}
if $2|E(C)\cap M|=|E(C)|$.  Let $\nu(G,M)$ denote the maximum number
of pairwise disjoint $M$-alternating circuits, and let $\tau (G,M)$ 
denote the minimum number of edges whose deletion leaves no 
$M$-alternating circuit.  It is clear that $\nu (G,M)=\nu (D(G,M))$
and $\tau (G,M)=\tau (D(G,M))$.  Thus we have the following corollary.

\begin{CO}\label{coro2}
Let $G$ be a brace, and let $M$ be a perfect matching
in $G$.  Then the following three conditions are equivalent.
\begin{itemize}
\item[\rm (i)] $G$ does not contain $K_{3,3}$ or the Heawood graph,
\item[\rm(ii)] $\tau(G',M')=\nu(G',M')$ for every subgraph $G'$ of $G$ such
that $M'=M\cap E(G')$ is a perfect matching in $G'$, and
\item[\rm (iii)] $G$ can be obtained from planar braces by repeatedly
applying the trisum operation.
\end{itemize}
\end{CO}

In fact, the equivalence of (i) and (ii) holds for all bipartite graphs,
not just braces.  We conclude this section with a lemma that will be needed
later.  The lemma follows immediately from \cite[Theorem~8.2]{Robertson:99}.
We say that a graph is a {\em cube} if it is isomorphic to the 
1-skeleton of the 3-dimensional cube.  Thus every cube has 8 
vertices and 12 edges.

\begin{LE}\label{lemmaX}
Let $G$ be a trisum of $G_1,G_2,G_3$ along $C$, where the graphs 
$G_1, G_2,G_3$ are obtained from planar braces by repeatedly applying
the trisum operation.  Then for $i=1,2,3$ we have $|E(G_i)|\ge 12$ with
equality if and only if $G_i$ is a cube.
\end{LE}
The remainder of the paper is dedicated to proving Proposition~\ref{mainprop}.
Consider $D,G,C$ as in Proposition~\ref{pr4.8},
and let $k$ be the number of edges of $M$ with both ends in $V(C)$.
As $M$ is a perfect of matching of $G$, $k\in\{0,1,2\}$.
Proposition~\ref{pr5.2} proves that $k\neq 2$,
Proposition~\ref{prkis1} proves that $k\neq 1$,
and finally
Proposition~\ref{pr7.2} proves that $k\neq 0$.
\section{Trisum-part I} \label{sec-2sum}
Let $D,G,M,G_1,G_2,G_3,C$ be as in Proposition~\ref{pr4.8}.  For
$i=1,2,3$ let $M'_i$ be the set of edges $M\cap E(G_i)$ with at least
one end not in $V(C)$, let $M_0$ be the set of edges of $C$ that are parallel
to an edge of $M$, and let $M_i=M'_i\cup M_0$.  We say that
$M_i$ is the {\em imprint} of $M$ on $G_i$.  
\begin{PR}\label{pr5.1}
Let a bipartite graph $G$ be a $4$-sum of $G_1$ and $G_2$ along $C$, let
$M$ be a perfect matching in $G$ such that some two edges of $M$ have
both ends in $V(C)$, let $D=D(G,M)$, and for $i=1,2$ let $M_i$ be the
imprint of $M$ on $G_i$.  If both $D(G_1,M_1)$ and $D(G_2,M_2)$ pack, 
then $\nu (D)=\tau(D)$.
\end{PR}
\begin{proof} For $i=1,2$ let $D_i=D(G_i,M_i)$.  Then $|V(D_1)\cap 
V(D_2)|=2$; let $V(D_1)\cap V(D_2)=\{u_1,u_2\}$.  Moreover, $E(D_1)\cap
E(D_2)=\{e_1,e_2\}$, where $e_1$ has head $u_2$ and tail $u_1$, and
$e_2$ has head $u_1$ and tail $u_2$.  For $i=1,2$ let 
$D'_i=D_i\backslash \{e_1,e_2\}$.  
\begin{claim} 
For each $D'_i$ ($i=1,2$) one of the following holds:
\begin{enumerate}
	\item[(1)]
There exists a maximum packing not using any of $u_1$ or $u_2$.
Every minimum \VFS\ does not contain any of $u_1$ or $u_2$.
	\item[(2)]
For some $k \in \{1,2\}$ the following holds:
all maximum packings use $u_k$,
there exists a maximum packing not using $u_{3-k}$,
and there exits a minimum \VFS\  which contains $u_k$ but not $u_{3-k}$.
	\item[(3)]
There exists a maximum packing using both $u_1$ and $u_2$.
There exists a minimum \VFS\  using $u_1$ and a minimum \VFS\  using $u_2$.
Moreover, either:
(a) there is a 
minimum
\VFS\ containing both $u_1,u_2$; or
(b) there is a packing of size $\tau(D'_i)-1$ not using $u_1$ or $u_2$.
\end{enumerate}
\end{claim}
\begin{cproof}

Observe that for (1)-(3) the statements about \VFS s  
(except for the last
sentence)
follow from the statements about maximum packings and Remark~\ref{easy-re}.
Suppose (1) does not hold; then every maximum packing of $D'_i$ uses one of $u_1,u_2$.
In particular $\nu(D_i)=\nu(D'_i)$. 
Suppose for a contradiction there exists 
a maximum packing $\zC_i$ of $D'_i$ not using $u_1$ and 
a maximum packing $\zC'_i$ of $D'_i$ not using $u_2$.
Remark~\ref{easy-re} implies that no
minimum \VFS\  of $D'_i$ contains $u_1$ or $u_2$.
Since $\{e_1,e_2\}$ is the edge-set of a circuit of $D_i$ 
this implies $\tau(D_i) > \tau(D'_i)$,
a contradiction since $D_i$ packs. 
Thus for some $k \in \{1,2\}$ every maximum packing of $D'_i$ uses $u_k$.
If (2) does not hold, then all maximum packings use $u_{3-k}$.
If (3)(a) does not hold, no minimum \VFS\  of $D'_i$ uses both $u_1$ and $u_2$.
This implies $\tau(D'_i \backslash \{u_1,u_2\}) \geq \tau(D'_i)-1$.
Since $D_i$ packs (3)(b) must hold.
\end{cproof}
\begin{claim} 
For $i=1,2$, let $T_i$ be a minimum \VFS\  of $D'_i$ and let $\zC_i$ be a maximum
packing of $D'_i$. We can assume one of the following holds:
\begin{itemize}
	\item[(a)]
There exists $k \in \{1,2\}$ such that 
$\zC_1$ and $\zC_2$ use $u_k$ but $u_k \not\in T_1 \cap T_2$.
	\item[(b)]
$\{u_1,u_2\} \cap (T_1 \cup T_2) = \emptyset$.
\end{itemize}
\end{claim}
\begin{cproof}
Let $T:=T_1 \cup T_2$ and let $\zC$ be an inclusion-wise maximal packing 
in $\zC_1 \cup \zC_2$.
If (a) does not hold, then $|T| \leq |\zC|$. 
If (b) does not hold, then $\{u_1,u_2\} \cap T \neq \emptyset$;
thus $T$ is a \VFS\  of $D$.
It follows that $\tau(D)=\nu(D)$, as desired.  
Thus we may assume that (a) or (b) holds.
\end{cproof}
We can assume, because of Claim~1 and Claim~2, that $D_1,D_2$ either 
both satisfy condition (1) of Claim~1, or they
both satisfy condition (3) of Claim~1 and one of $D'_1,D'_2$, say $D'_1$, satisfies (3)(b).
Consider the latter case first.
Let $T_1$ (resp. $T_2$) be a minimum \VFS\  of $D'_1$ 
(resp. $D'_2$) using $u_1$.
Let $T:=T_1 \cup T_2$. 
Let $\zC_1$ be a packing of $D'_1 \backslash \{u_1,u_2\}$ of size 
$\tau(D'_1)-1$
and let $\zC_2$ be a maximum packing of $D'_2$.
Clearly $\zC:=\zC_1 \cup \zC_2$ is a packing in $D$.
Since $|T_1 \cup T_2| = \tau(D'_1) + \tau(D'_2)-1$ and 
$|\zC|= \tau(D'_1)-1+\tau(D'_2)$, we have $\tau(D)=\nu(D)$.

Thus we may assume that both $D'_1,D'_2$ satisfy (1).
For $i=1,2$, let $\zC_i$ be a maximum packing of $D_i$.
Suppose there is $k \in \{1,2\}$ such that for $i=1,2$, 
$\tau(D'_i+u_ku_{3-k})=\tau(D'_i)$ and let $T_i$ be the corresponding minimum \VFS.
Then $T_i$ intersects all $u_{3-k}u_k$-paths of $D_i$.
Hence $T:=T_1 \cup T_2$ is a \VFS\  of $D$.
Moreover, $|T|=\tau(D'_1)+\tau(D'_2)=|\zC_1 \cup \zC_2|$, i.e. $\tau(D)=\nu(D)$.
Thus we can assume there is for $k=1,2$ an index $t(k) \in \{1,2\}$ 
such that $\tau(D'_{t(k)}+u_ku_{3-k})>\tau(D'_{t(k)})$.
Since $D_1,D_2$ pack $\nu(D'_{t(k)}+u_ku_{3-k})>\tau(D'_{t(k)})$;
let $\zF_{t(k)}$ be the corresponding packing.
Some circuit $C_{t(k)}$ of $\zF_{t(k)}$ is of the form $P_{t(k)}+u_ku_{3-k}$ 
where $P_{t(k)}$ is a $u_{3-k}u_k$-path.
For $i=1,2$ let $T_i$ be a minimum \VFS\  of $D'_i$.
Note that $T_{t(k)}$ does not intersect $P_{t(k)}$.
Observe that we cannot have $t(1)=t(2)=i \in \{1,2\}$, 
for otherwise there exist both a $u_1u_2$- and $u_2u_1$-paths in $D'_i$ 
which are not intersected by $T_i$ and hence $T_i$ does not
intersect all circuits of $D'_i$, a contradiction.
Thus we can assume $t(1)=1$ and $t(2)=2$.
Let $\zC:=\zF_1 \cup \zF_2 \cup \{P_1 \cup P_2\}-\{C_1,C_2\}$.
Then $\zC$ is a packing of $D$ and $T:=T_1 \cup T_2 \cup \{u_1\}$ is a \VFS\  of $D$.
Moreover, $|T|=\tau(D'_1)+\tau(D'_2)+1=|\zC|$, i.e. $\tau(D)=\nu(D)$, 
as desired.
\end{proof}

\begin{PR}\label{pr5.2} 
Let $G, M, D$, where $\nu(D)<\tau(D)$, and $G_1,G_2,G_3,C$ be as in 
Proposition~\ref{pr4.8}. Then at most one edge of $M$ has
both ends in $V(C)$.
\end{PR}
\begin{proof}
Suppose for a contradiction that two edges of $M$ have both ends
in $V(C)$.
For $i=1,2,3$ let $M_i$ be the imprint of $M$ on $G_i$.  The graphs 
$G_1$ and $G_2\cup G_3$ are obtained from planar braces by repeatedly
applying the 4-sum operation, and hence the digraphs $D_1=D(G_1,M_1)$ and
$D_2=D(G_2\cup G_3, M_2\cup M_3)$ have no minor isomorphic to an 
odd double circuit or $F_7$ by Theorem~\ref{maintool}.  
Thus $D_1$ and $D_2$ pack, and hence
by Proposition~\ref{pr5.1} $\nu (D)=\tau (D)$, a contradiction.
\end{proof}
%
%
\section{Trisum-part II} \label{sec-3sum}

\begin{LE}\label{lem6.1}
 Let $D_1,D_2$ be digraphs with $V(D_1)\cap V(D_2)=
\{u_1,u_2,u_3\}$ and $E(D_1)\cap E(D_2)=\emptyset$.  Let $D=D_1\cup D_2$,
$a\not\in V(D)$, 
$E_1=\{u_1u_2, u_1 u_3, u_2u_3\}$, $E_2=\{u_2u_1,u_3u_1,
u_3u_2\}$, 
$Z_1=\{au_2,u_2a,u_1a,au_3\}$, and 
$Z_2=\{au_2,u_2a,au_1,\allowbreak u_3a\}$,
where $a\not\in V(D)$.
Assume that
\begin{itemize}
\item[\rm(a)] if, for $i=1,2$, $C_i$ is a circuit of $D_i$,
then $V(C_1)\cap V(C_2)\subseteq\{u_2\}$,
\item[\rm(b)] if $C$ is a circuit of $D$ that uses edges of both $D_1$ and
$D_2$, then $C=P_1\cup P_2$ and there exist integers $i,j\in \{1,2,3\}$
such that $i<j$ and $P_1$ is a $u_ju_i$-path of $D_1$ and $P_2$ is a 
$u_iu_j$-path of $D_2$, and
\item[\rm(c)] there exist integers $i,j$ such that $\{i,j\}=\{1,2\}$,
$D_i+E_i$ packs and is strongly $2$-connected, 
and $D_j+Z_j$ packs.  
\end{itemize}
\noindent Then $\tau (D)=\nu(D)$.
\end{LE}

\begin{proof} Suppose for a contradiction that $\nu (D)<\tau (D)$.
\begin{claim}\label{clm7.1}
The digraph $D$ has a packing of size $\nu (D_1)+\nu(D_2)-1$.
\end{claim}

\begin{cproof}
Clearly $\nu (D_2\backslash u_2)\ge\nu(D_2)-1$, and so the union of any
maximum packing of $D_1$ with any packing of $D_2\backslash u_2$ of
size $\nu(D_2)-1$ is as desired by (a).  This proves Claim~\ref{clm7.1}.
\end{cproof}
\begin{claim}\label{clm7.2}
The digraph $D$ has a \VFS\  of size at most $\tau (D_1)+\tau (D_2)+1$.
\end{claim}

\begin{cproof}
By (c) we may assume from the symmetry that $D_1+E_1$ packs.  Clearly
$\nu (D_1+E_1)\le\nu(D_1)+1$.  Thus $\tau(D_1+E_1)\le \tau (D_1)+1$.
Let $T_1$ be a \VFS\  of $D_1+E_1$ of size at most $\tau(D_1)+1$, and
let $T_2$ be a \VFS\  of $D_2$ of size $\tau(D_2)$.  By (b) $T_1\cup T_2$ is
a \VFS\  of $D$, as required.  This proves Claim~\ref{clm7.2}.
\end{cproof}

For $i=1,2$ let $F_i$ be the set of all edges $f\in E_i$ such that
$\nu (D_i+f)=\nu (D_i)$.
\begin{claim}\label{clm7.3}
For $i=1,2$, $\nu(D_i+F_i)=\nu (D_i)$.
\end{claim}

\begin{cproof}
If $\nu (D_i+F_i)>\nu(D_i)$, then, since every edge of $E_i$ has both ends
in $\{u_1,u_2,u_3\}$, we deduce that $\nu(D_i+f)>\nu(D_i)$ for some
$f\in F_i$, a contradiction.  This proves Claim~\ref{clm7.3}.
\end{cproof}
\begin{claim}\label{clm7.4}
Let $i,j\in\{1,2,3\}$ be such that $i<j$, and let $D'$ be a subdigraph of
 $D_1$.  If $\nu(D'+u_iu_j)>\nu(D')$, then there exist a maximum packing
$\cal C$ of $D'$ and a path $P$ in $D'$ from $u_j$ to $u_i$ such that 
every member of $\cal C$ is disjoint from $P$.
\end{claim}

\begin{cproof}
Let $\cal C'$ be a maximum packing of $D'+u_iu_j$.  Since $\cal C'$ is not a 
packing of $D'$, some member of $\cal C'$, say $C$,
 uses the edge $u_iu_j$. Thus $\cal C'-\{C\}$ and $C\backslash u_iu_j$ satisfy
the conclusion of the claim.
\end{cproof}
\begin{claim}\label{clm7.5}
If $D_1+E_1$ packs and every maximum packing of $D_1+u_1u_3$ uses $u_2$, then
every maximum packing of $D_1$ uses $u_2$.
\end{claim}
\begin{cproof}
Suppose for a contradiction that every maximum packing of $D_1+u_1u_3$ uses
$u_2$, but some maximum packing of $D_1$ does not use $u_2$.  Then
$\nu (D_1+u_1u_3)>\nu(D_1)$.  By Claim~\ref{clm7.4} applied to $D'=D_1$
there exist a maximum packing $\cal C$ of $D_1$ and a path $P$ of
$D_1$ from $u_3$ to $u_1$ such that $P$ is disjoint from every member of
$\cal C$.  Let 
$L$
be a subdigraph of $D_1$ such that
\begin{itemize}
\item[$(\alpha)$] $L$ includes $P$ and every member of $\cal C$,
\item[$(\beta)$] $L$ includes every member of some maximum packing of $D_1$ that does not use $u_2$, and
\item[$(\gamma)$] subject to $(\alpha)$ and $(\beta)$, $E(L)$ is minimal.
\end{itemize}
By $(\alpha)$ $\nu(L)=\nu(D_1)$.  
We claim that $\nu(L+u_1u_2+u_2u_3)>\nu(L)$.  
To prove this claim suppose for a contradiction that equality holds.  
Since $D_1+E_1$ packs we deduce that $\tau(L+u_1u_2+u_2u_3)=\nu(L)$.  
Let $T$ be a \VFS\  of $L+u_1u_2+u_2u_3$ of size $\nu(L)$.  
From $(\beta)$ we deduce that $u_2\not\in T$, but then it follows that $T$ is a \VFS\  
of $L+u_1u_3$, contrary to $(\alpha)$.  
This proves that $\nu(L+u_1u_2+u_2u_3)>\nu(L)$.  
Let $\cal S$ be a maximum packing of $L+u_1u_2+u_2u_3$.
We may assume that no member $C$ of $\cal S$ uses both edges $u_1u_2,u_2u_3$,
 for otherwise 
$\cal S\setminus\{C\}\cup\{C+u_1u_3-u_1u_2-u_2u_3\}$ is a maximum packing of $D_1+u_1u_3$ avoiding $u_2$, a contradiction.
Hence, either $\nu(L+u_1u_2)>\nu(L)$ or $\nu(L+u_2u_3)>\nu(L)$, and so we may assume the former.  
By Claim~\ref{clm7.4} 
applied to $D'=L$
there exists a maximum packing $\cal C'$ of $L$ and a
path $P'$ in $L$ from $u_2$ to $u_1$ disjoint from every member of $\cal C'$.
Since the union of $P$ and all members of $\cal C$ does not include a path from
$u_2$ to $u_1$, there exists an edge $e\in E(P')$ that does not belong to $P$
or any member of $\cal C$.  Thus $L\backslash e$ satisfies $(\alpha)$.
But $L\backslash e$ includes every member of $\cal C'$, and hence it also
satisfies $(\beta)$, contrary to $(\gamma)$.  This proves Claim~\ref{clm7.5}.
\end{cproof}

\begin{claim}\label{clm7.6}
Let $i,j\in \{1,2,3\}$ with $i<j$.  If $u_iu_j\not\in F_1$, then 
$u_ju_i\in F_2$.
\end{claim}

\begin{cproof}
Suppose for a contradiction that $u_iu_j\not\in F_1$ and $u_ju_i\not\in F_2$.
Let $\cal C_1$ be a packing of $D_1+u_iu_j$ of size $\nu(D_1)+1$, and
let $\cal C_2$ be a packing of $D_2+u_ju_i$ of size $\nu(D_2)+1$.  Then
$\cal C_1$ includes a circuit $C_1$ containing $u_iu_j$, and
$\cal C_2$ includes a circuit $C_2$ containing $u_ju_i$. 
Let $C$ be the circuit $(C_1\backslash u_iu_j)\cup(C_2\backslash
u_ju_i)$.  If one of $\cal C_1-\{C_1\}$, $\cal C_2-\{C_2\}$ does not use
$u_2$, then $\cal C:=(\cal C_1-\{C_1\})\cup(\cal C_2-\{C_2\})\cup\{C\}$ is a 
packing of $D$ of size $\nu(D_1)+\nu(D_2)+1$ by (a).
Then because of Claim~\ref{clm7.2}, $\tau(D)=\nu(D)$ packs, a contradiction.
Thus we may assume that both $\cal C_1-\{C_1\}$, $\cal C_2-\{C_2\}$ use
$u_2$ for all choices of $\cal C_1$ and $\cal C_2$.  Thus $i=1$ and $j=3$,
and every maximum packing of $D_1+u_1u_3$ or $D_2+u_3u_1$ uses $u_2$.
By (c) we may assume that $D_1+E_1$ and $D_2+Z_2$ packs.
Hence by Claim~\ref{clm7.5}
every maximum packing of $D_1$ uses $u_2$.  By Remark~\ref{easy-re}
$D_1$ has \VFS\  $T_1$ of size $\nu(D_1)$ with $u_2\in T_1$, and
$D_2+u_3u_1$ has a \VFS\  $T_2$ of size $\nu(D_2)+1$ with
$u_2\in T_2$.  By (b) $T_1\cup T_2$ is a \VFS\  of $D$ of size $\nu(D_1)+
\nu(D_2)$.  On the other hand, by deleting one of the circuits of $\cal C$
that contain $u_2$ we obtain a packing of $D$ of size $\nu(D_1)+\nu(D_2)$.
Thus $\nu(D)=\tau(D)$,
a contradiction.  This proves Claim~\ref{clm7.6}.
\end{cproof}

\begin{claim}\label{clm7.7}
The digraph $D$ has a packing of size $\nu(D_1)+\nu(D_2)$.
\end{claim}

\begin{cproof}
Suppose not.  Then for $i=1,2$ every maximum packing of $D_i$ uses $u_2$,
for otherwise the union of a maximum packing in $D_i$ that does not use
$u_2$ with any maximum packing of $D_{3-i}$ is as desired.  By 
Remark~\ref{easy-re} the digraph $D_i$ has a \VFS\  $T_i$ of size $\tau(D_i)$ with
$u_2\in T_i$.  Let us assume first that $\nu(D_1+u_1u_3)>\nu(D_1)$.
Then $\nu(D_2+u_3u_1)=\nu(D_2)$ by Claim~\ref{clm7.6}.
The graph $D_2+u_3u_1$ packs (because by (c) $D_2+E_2$ or $D_2+Z_2$ packs), 
and so $\nu(D_2+u_3u_1\backslash u_2)=
\tau(D_2+u_3u_1\backslash u_2)$.  If $\nu(D_2+u_3u_1\backslash u_2)=
\nu(D_2)$, then let $\cal C_1$ be a maximum packing in $D_1+u_1u_3$ and 
let $\cal C_2$ be a maximum packing in $\nu(D_2+u_3u_1\backslash u_2)$.
Then some circuit of $\cal C_1$ uses the edge $u_1u_3$ (because
$\nu(D_1+u_1u_3)>\nu(D_1)$), and some circuit of $\cal C_2$ uses the edge
$u_3u_1$ (because every maximum packing of $D_2$ uses $u_2)$.
Thus $\cal C_1$ and $\cal C_2$ can be combined as in the proof of 
Claim~\ref{clm7.6} to produce the desired packing of $D$.  Thus we may 
assume that $\nu(D_2+u_3u_1\backslash u_2)<\nu(D_2)$.  Let $T'_2$ be a \VFS\ 
in $D_2+u_3u_1\backslash u_2$ of size $\nu(D_2)-1$;
then $T_1\cup T'_2$ is a \VFS\  in $D$ by (b), and its size is
$\nu(D_1)+\nu(D_2)-1$, contrary to Claim~\ref{clm7.1}.
This completes the case when $\nu(D_1+u_1u_3)>\nu(D_1)$.  

Thus we may assume
that $\nu(D_1+u_1u_3)=\nu(D_1)$ and $\nu(D_2+u_3u_1)=\nu(D_2)$.  From the
symmetry and (c) we may assume that $D_2+Z_2$ packs.  Since every
maximum packing of $D_2$ uses $u_2$, and $\nu(D_2+u_3u_1)=\nu(D_2)$,
we see that $\nu(D_2+Z_2)=\nu(D_2)$.  Since $D_2+Z_2$ packs, there exists
a \VFS\  $T''_2$ of $D_2+Z_2$ of size $\tau(D_2)$.  Since
$T''_2\cap V(D_2)$ is a \VFS\  of $D_2$, we deduce that $a\not\in T''_2$, and
hence $u_2\in T''_2$, because $T''_2$ intersects the circuit of
$D_2+Z_2$ with vertex-set $\{a,u_2\}$.  
Thus $T''_2$ is a \VFS\  of $D_2+u_3u_1$ with $u_2\in T''_2$, and 
so $T_1\cup T''_2$ is a \VFS\  of $D$ by (b).  Moreover, $|T_1\cup T''_2|=
\tau (D_1)+\tau(D_2)-1$, contrary to Claim~\ref{clm7.1}.  This completes
the proof of Claim~\ref{clm7.7}.
\end{cproof}

We are now ready to complete the proof of the lemma.  We claim that one
of $D_1+F_1$, $D_2+F_2$ does not pack.  Indeed, if both of them pack, then
by Claim~\ref{clm7.3} the digraph $D_i+F_i$ has a \VFS\  of size $\nu(D_i)$,
and the union of those sets is a \VFS\  in $D$ by (b) of size $\nu(D_1)+
\nu(D_2)$, contrary to Claim~\ref{clm7.7}.  Thus we may assume that $D_2+
F_2$ does not pack.

By (c) the digraph $D_1+E_1$ packs and is strongly 2-connected, 
and $D_2+Z_2$ packs.  To motivate the next step, notice that
since $D_2+Z_2$ packs, but $D_2+F_2$
does not, we have $u_2u_1,u_3u_2\in F_2$.  Since $D_1+E_1$ packs,
so does $D_1+F_1$, and hence by Claim~\ref{clm7.3} there exists a \VFS\  $T_1$
in $D_1+F_1$ of size $\tau(D_1)$.  

We claim that the
set $T_1$ is a \VFS\  in $D_1+F_1+u_1u_2$ or $D_1+F_1+u_2u_3$.
To prove this claim suppose for a contradiction that this is not the case.
We deduce that there exist a $u_2u_1$-path
$P_1$ and a $u_3u_2$-path $P_2$ in $D_1$, both disjoint from $T_1$.
Since $T_1$ intersects every circuit of $D_1$, it follows that 
$V(P_1)\cap V(P_2)=\{u_2\}$.  Since $D_1+E_1$ is strongly 
2-connected, there exists a path $Q$ in $D_1$ from $V(P_2)-\{u_2\}$ to
$V(P_1)-\{u_2\}$; we may assume that no interior vertex of $Q$
belongs to $V(P_1)\cup V(P_2)$.  Let $H$ be the digraph $P_1\cup P_2
\cup Q+E_1$; then $\nu (H)=1<2=\tau (H)$, contrary to the fact that
$D_1+E_1$ packs.  This proves our claim that 
$T_1$ is a \VFS\  in $D_1+F_1+u_1u_2$ or $D_1+F_1+u_2u_3$.

From the symmetry we may assume that $T_1$ is a \VFS\  in $D_1+F_1+u_1u_2$.
Let $F'_2=F_2-\{u_1u_2\}$.  Since $D_2+Z_2$ packs, so does its minor
$D_2+F'_2$, and so by Claim~\ref{clm3} the digraph $D_2+F'_2$ has a \VFS\ 
$T_2$ of size $\tau(D_2)$.  By (b) the set $T_1\cup T_2$ is a \VFS\  in 
$D$, and its size is $\tau (D_1)+\tau (D_2)$, contrary to Claim~\ref{clm7.7}.\
\end{proof}

\begin{PR}\label{prkis1}
Let $G,M,D$, where $\nu(D)<\tau(D)$, and $G_1,G_2,G_3,C$ be as in 
Proposition~\ref{pr4.8}.
Then either none or exactly two edges of $M$ have both ends in $V(C)$.
\end{PR}
\begin{proof}
Let $A,B$ denote a  bipartition of $G$.
Let $v_1,v'_2,v_2,v_3$ be the vertices of $C$ (in that order), where 
$v_1,v_2 \in A$.
For $i=1,2,3$ let $m_i$ be the edge of $M$ incident with $v_i$. 
Suppose for a contradiction that $m_2$ is the only edge of $M$ with
both ends in $V(C)$.
We may assume that $m_2$ is incident with $v'_2$.
Thus $m_1,m_3$ are distinct and are incident with vertices not on $C$.
We may also assume that $m_1,m_3\in E(G_1)\cup E(G_2)$. 
For $i=1,2,3$ let $M_i$ be the imprint of $M$ on $G_i$ (see the paragraph
prior to Proposition~\ref{pr5.1} for a definition).
Let $J_1:=D(G_1\cup G_2,M_1\cup M_2)$, let $Q$ be a 
cube such that $C$ is a subgraph of $Q$ and otherwise
$Q$ is disjoint from $G_3$, and let $J_2:=D(G_3\cup Q,M'_3)$, where $M'_3$
is a perfect matching of $G_3\cup Q$ with $M_3\subseteq M'_3$ 
that does not use an edge joining $v_1$ and $v_3$.  Such a matching is unique, 
and it has a unique element, say $m_0$, not incident with a vertex of $G_3$.
Let $a$ denote the vertex of $J_2$ that results from contracting $m_0$,
and in both $J_1,J_2$ let $u_1,u_2,u_3$ denote the vertices that result from
contracting the edges incident with $v_1,v_2,v_3$, respectively.  

Let $D_1$ be obtained from $J_1$ by deleting the edges of $C$, and let $D_2$
be obtained from $J_2$ by deleting the vertex $a$ and edges of $Q\cup C$.
We wish to apply Lemma~\ref{lem6.1} to the digraphs $D_1$ and $D_2$.
Since $u_1$ is a source and $u_3$ is a sink of $D_2$, we see
immediately that (a) and (b) of that lemma hold.  We will show that
$i=1$ and $j=2$ satisfy (c).  
Since $G_1$ and $G_2$ are braces, so is $G_1\cup G_2$, and thus $J_1$ is
strongly 2-connected by Proposition~\ref{connectivity}.  
To show that $D_1+E_1$ packs we first notice that $D_1+E_1$ is
isomorphic to $J_1$.  But $G_1\cup G_2$ is obtained from planar braces by
repeatedly applying the trisum operation, and hence $J_1$ has no odd
double circuit or $F_7$ minor by Theorem~\ref{maintool}.  Moreover,
$|V(J_1)|+|E(J_1)|=|E(G_1\cup G_2)|<|E(G)|=|V(D)|+|E(D)|$ by
Lemma~\ref{lemmaX}, and hence $J_1$ (and thus $D_1+E_1$) pack by
the hypothesis of Proposition~\ref{pr4.8}.
Finally, $D_2+Z_2$ is a subdigraph of $J_2$, and hence it packs, by the 
argument of this paragraph.
Thus $\nu (D)=\tau (D)$ by Proposition~\ref{lem6.1}, a contradiction.  
\end{proof}
%
%
\section{Trisum-part III} \label{sec-4sum}
Let $D_1,D_2$ be edge-disjoint subdigraphs of a digraph $D$, let
$X\subseteq V(D_1)\cap V(D_2)$, and let $C$ be a circuit of $D$. 
We say that $C$ {\em passes from $D_1$ to $D_2$ through} $X$ if there
is no vertex $v\in V(D)-X$ such  that the edge of $C$ with head $v$ belongs
to $D_1$ and the edge of $C$ with tail $v$ belongs to $D_2$.

\begin{LE} \label{4sum-le}
Let $D_1$ and $D_2$ be digraphs with 
$V(D_1) \cap V(D_2)=\{u_1,u_2,u_3,u_4\}$ and
$E(D_1) \cap E(D_2)=\emptyset$.
Let $D=D_1\cup D_2$, 
let $E_1=\{u_1u_2,u_3u_2,u_3u_4,u_1u_4\}$, and 
let $E_2=\{u_2u_1,u_2u_3,u_4u_3,u_4u_1\}$.
Assume that 
\begin{itemize}
\item[\rm(1)] for $i=1,2$, $D_i+E_i$ packs, 
\item[\rm(2)] every circuit of $D_1$ is disjoint from every circuit of $D_2$,
\item[\rm(3)] every circuit of $D$ passes from $D_1$ to $D_2$ through
$\{u_1,u_3\}$, and it passes from $D_2$ to $D_1$ through $\{u_2,u_4\}$.
\end{itemize}
Moreover, assume that
for every pair $e_1,e_2\in E_i$ of independent edges one of the following holds:
\begin{itemize}
	\item [(a)] $\nu(D_i+e_1+e_2)\ge\nu(D_i)+2$,
	\item [(b)] $\tau(D_i+e_1)=\tau(D_i)$, or
	\item [(c)] $\tau(D_i+e_2)=\tau(D_i)$.
\end{itemize}
Then $\tau (D)=\nu (D)$.
\end{LE}

\begin{proof} Suppose for a contradiction that $\nu (D)<\tau (D)$.

\begin{claim}\label{clm1} 
Let $i=1$ or $i=2$, and let $F\subseteq E_i$. Then one of the following holds:
\begin{enumi}
	\item There is an edge $e\in F$ such that $\nu(D_i+e)>\nu(D_i)$,
	\item $\tau(D_i+F)=\tau(D_i)$, or
	\item there exist independent edges $e_1,e_2\in F$ such that
	      \center{$\nu(D_i)=\nu(D_i+e_1)=\nu(D_i+e_2)<\nu(D_i+e_1+e_2)$}.
\end{enumi}
\end{claim}
\begin{cproof}
Suppose (ii) does not hold, i.e. $\tau(D_i+F)>\tau(D_i)$.
As $D_i+E_i$ packs, $\nu(D_i+F) > \nu(D_i)$.
Now if (i) does not hold then (iii) must hold since 
if two edges $e_1,e_2 \in F$ appear in the same circuit then $e_1,e_2$ are independent.
\end{cproof}
\begin{claim}\label{clm2} 
$D$ has a \VFS\ of size $\nu(D_1)+\nu(D_2)+1$.
\end{claim}
\begin{cproof}
If $\tau(D_i+E_i)\le \tau(D_i)+1$ for some $i \in \{1,2\}$, 
then take the corresponding \VFS, and union it with any \VFS\  
of $D_{3-i}$ of size $\tau(D_{3-i})$.
The resulting set is a \VFS\  in $D$
of size $\nu(D_1)+\nu(D_2)+1$ by (3), as desired.
Thus we may assume that $\tau(D_i+E_i)\ge\tau(D_i)+2$ for $i=1,2$. 
Since $\nu(D_i+E_i)=\tau(D_i+E_i)$ we may assume that there is a packing of 
size $\nu(D_1)$ in $D_1$ and two disjoint paths disjoint from the packing 
joining $u_2$ to $u_3$ and $u_4$ to $u_1$, respectively.
Likewise, we may assume that a similar situation occurs in $D_2$,
but with paths joining $u_3$ to $u_4$ and $u_1$ to $u_2$. (If the paths join
the other pairs we get a packing of size $\nu(D_1)+\nu(D_2)+2$, a
contradiction, because the union of $\{u_1,u_3\}$, any \VFS\ of $D_1$ 
and any \VFS\ of $D_2$ is a \VFS\ of $D$ of the same size.)
Now we use the fact that $D_2$ satisfies (a), (b) or (c) for 
the edges $u_2u_3$ and $u_4u_1$. 
If (a) holds, then we have a packing in $D$ of size $\nu(D_1)+\nu(D_2)+2$, 
and so we
may assume from the symmetry that (b) holds, where $e_1=u_2u_3$. Let $T_2$ be
the corresponding \VFS. 
We may also assume that $\nu(D_1+u_3u_4+u_1u_2)\le
\nu(D_1)+1$, for otherwise we produce a packing of $D$ 
of size $\nu(D_1)+\nu(D_2)+2$.
It follows that $\nu(D_1+u_3u_4+u_1u_2+u_1u_4) \le \nu(D_1)+1$,
because a packing of $D_1+u_3u_4+u_1u_2+u_1u_4$ that uses $u_1u_4$ 
cannot use $u_3u_4$ or $u_1u_2$.
Hence $\tau(D_1+u_3u_4+u_1u_2+u_1u_4) = 
\nu(D_1+u_3u_4+u_1u_2+u_1u_4)\le\tau(D_1)+1$.
Let $T_1$ be a corresponding \VFS. 
Then $T_1\cup T_2$ is a \VFS\  in $D$
of size $\nu(D_1)+\nu(D_2)+1$ by (3), as desired.
\end{cproof}

Let $F_i$ be the set of all edges $e\in E_i$ such that $\tau(D_i+e)>\tau(D_i)$.
\begin{claim}\label{clm3} 
The reversal of no edge in $F_1$ belongs to $F_2$.
\end{claim}
\begin{cproof}
Otherwise we can construct a packing in $D$ 
of size $\nu(D_1)+\nu(D_2)+1$, contrary to Claim~2.
\end{cproof}

\begin{claim}\label{clm4} 
The digraph $D$ has a packing of size $\nu (D_1)+\nu(D_2)$.\end{claim}
\begin{cproof}
The union of any maximum packing of $D_1$ with any maximum packing of 
$D_2$ is as desired by (2).
\end{cproof}

\begin{claim}\label{clm4+1} 
For some $i\in\{1,2\}$, $F_i$ includes two independent edges.
\end{claim}

\begin{cproof}
Suppose for a contradiction that no $F_i$ includes two
independent edges. It follows from Claim~\ref{clm3} 
that there exist adjacent edges
$e_1,e_2\in E_1-F_1$ and adjacent edges $e_3,e_4\in E_2-F_2$ such that
$e_3,e_4$ are the reverses of the edges in $E_1-\{e_1,e_2\}$. Since
$e_1,e_2\not\in F_1$ we deduce from Claim~\ref{clm1} that $\tau(D_1+e_1+e_2)=
\nu(D_1+e_1+e_2)=\nu(D_1)$ and similarly $\tau(D_2+e_3+e_4)=\nu(D_2)$.
But the union of the corresponding \VFS s  
is a \VFS\  in $D$ of size $\nu(D_1)+\nu(D_2)$, 
contrary to Claim~\ref{clm4}.
\end{cproof}
\begin{claim}\label{clm5} 
At most one of $F_1$, $F_2$ includes two independent edges.
\end{claim}
\begin{cproof}
If both of them do, then (a) holds for those pairs, and we get a packing
in $D$ of size at least $\nu(D_1)+\nu(D_2)+1$, contradicting Claim~\ref{clm2}.
\end{cproof}
\noindent
By Claim~\ref{clm4+1} we may assume that $F_2$ includes two independent edges.
We wish to define a set $F\subseteq E_1-F_1$.
If $E_2=F_2$, then $F_1=\emptyset$ by Claim~\ref{clm3}, and we put $F=E_1$.
Otherwise we proceed as follows. If $F_1\ne\emptyset$, then it includes
a unique edge by Claim~\ref{clm3}, Claim~\ref{clm5}
 and the fact that $F_2$ includes two independent
edges. Let $e$ be the unique member of $F_1$. If $F_1=\emptyset$,
then we select $e\in E_1$ such that its reverse does not belong to $F_2$.
In either case the reverse of $e$ does not belong to $F_2$. We put
$F=E_1-\{e\}$. This completes the definition of $F$.
We apply Claim~\ref{clm1} to $D_1$ and $F$. 
Then (i) does not hold, because $F\cap F_1=\emptyset$.
If (ii) holds, then let $T_1$ be the corresponding \VFS, and
let $T_2$ be a \VFS\  of size $\tau(D_2)$ in $D_2$ if $e$ does
not exist, and in $D_2$ with the reverse of $e$ added otherwise.
Then $T_1\cup T_2$ is a \VFS\  in $D$ by (3) of size $\nu(D_1)+\nu(D_2)$,
contrary to Claim~\ref{clm4}. 
Thus (iii) holds. That is, there exist independent
edges $e_1,e_2\in F$ such that $\nu(D_1+e_1+e_2)>\nu(D_1)$.
Let $e_3,e_4\in E_2$ be the reverses of $e_1,e_2$.
Since $F_2$ includes two independent edges we deduce from the choice of
$F$ that $e_3,e_4\in F_2$. Thus 
$\nu(D_2+e_3+e_4)\ge\nu(D_2)+2$
by (a).
By combining the resulting packings we get a packing in $D$ of size
at least $\nu(D_1)+\nu(D_2)+1$, contrary to Claim~\ref{clm2}.
\end{proof}
\begin{PR}\label{pr7.2}
Let $G,M,D$, where 
$\nu(D)<\tau(D)$,
and $G_1,G_2,G_3,C$ 
be as in Proposition~{\rm\ref{pr4.8}}.
Then at least one edge of $M$ has both ends in $V(C)$.
\end{PR}
\begin{proof}
Let $A,B$ denote a bipartition of $G$.
Let $u_1,u_2,u_3,u_4$ be the vertices of $C$ (in that order),
where $u_1,u_3 \in A$ and $u_2,u_4 \in B$.
Suppose for a contradiction that no edge of $M$ has both ends in $V(C)$,
and let the edges of $M$ incident to vertices of $C$ be 
$m_1=u_1u'_1$, $m_2=u_2u'_2$, $m_3=u_3u'_3$, $m_4=u_4u'_4$.
For $i=1,2,3,4$ we will use $u_i$ to also denote the vertex of $D$ that
results from contracting $m_i$.
Let $Q$ be a cube such that $C$ is a subgraph 
of $Q$, and $Q$ is otherwise disjoint from $G_1\cup G_2\cup G_3$. 
Since $G$ is a brace, $|V(G_i)\setminus\{u_1,\ldots,u_4\}|$ is even for $i=1,2,3,4$.
As each of $m_1,m_2,m_3,m_4$ have exactly one end in $C$,
we may assume (by renumbering $G_1,G_2,G_3$ and $u_1,u_2,u_3,u_4$) that
$\{m_1,m_2,m_3,m_4\}\subseteq E(G_1)$, or $\{m_3,m_4\}\subseteq E(G_1)$ and $\{m_1,m_2\}\subseteq E(G_2)$.  
In the former case we may also assume that
$|E(G_2)|\le |E(G_3)|$.  If $\{m_1,m_2,m_3,m_4\}\subseteq E(G_1)$ and
$|E(G_1)|>12$, then let $H_1=G_1$ and $H_2=G_2\cup G_3$; otherwise
let $H_1=G_1\cup G_2$ and $H_2=G_3$.  Thus $|E(H_1)|>12$ by Lemma~\ref{lemmaX}.
Then both $H_1$ and $H_2$ are obtained from planar braces by
repeatedly applying the trisum operation. Let 
$J_1=D(H_1,M)$, and let $D_1=J_1\backslash E(C)$.
Let $J_2$ be obtained from $H_2$ by directing every edge from $A\cap
V(H_2)$ to $B\cap V(H_2)$, and then contracting every edge of $M\cap E(H_2)$,
and let $D_2=J_2\backslash E(C)$.  Let us notice that $u_1,u_3$ are
sources, and $u_2,u_4$ are sinks of $D_2$.  Thus conditions (2) and (3)
of Lemma~\ref{4sum-le} hold.  

We now prove that condition (1) holds.  
The graph $H_1$ is obtained from planar braces by repeatedly applying
the 4-sum operation.  By Theorem~\ref{maintool} the digraph $J_1$ has
no minor isomorphic to an odd double circuit or $F_7$.  Moreover
$|V(J_1)|+|E(J_1)|<|V(D)|+|E(D)|$ by Lemma~\ref{lemmaX}, and so $J_1$ 
packs by the hypothesis of Proposition~\ref{pr4.8}.  But $J_1$ is
isomorphic to $D_1+E_1$, and hence $D_1+E_1$ packs.
To prove that $D_2+E_2$ packs
we first notice that $D_2+E_2$ is a subdigraph of $D(H_2\cup Q,M_2)$, where
$M_2$ is a perfect matching of $H_2\cup Q$ that includes $E(H_2)\cap M$ and
no edge with both ends in $V(C)$.  But $D(H_2\cup Q,M_2)$ packs by 
the hypothesis of Proposition~\ref{pr4.8} and the fact that $|E(H_1)|
 >12$.  Thus conditions (1)--(3) of Lemma~\ref{4sum-le} hold.

Next we show that for $i=1,2$, and for every pair $e_1,e_2\in E_i$ of
independent edges one of (a), (b), (c) holds.  We first do so
for $i=2$.  It suffices to argue for $e_1=u_2u_1$ and $e_2=u_4u_3$.
Since $D(H_2\cup Q,M_2)$ packs by the previous paragraph, we see that $D'_2=
D_2+\{u_2u_1,u_3u_2,u_4u_3,u_1u_4\}$ also packs.  But clearly $\tau 
(D'_2)>\tau (D_2)$, because $u_1,u_3$ are sources and $u_1,u_4$ are
sinks of $D_2$, and $\{u_1,u_2,u_3,u_4\}$ is the vertex-set
of a circuit of $D'_2$.  If $\nu (D'_2)\ge\nu(D_2)+2$, then (a) holds.
Thus we may assume that $\tau (D'_2)=\tau (D_2)+1$.  Let $T$ be a
corresponding \VFS\ of $D_2'$.  Since $\{u_1,u_2,u_3,u_4\}$ is the vertex-set
of a circuit of $D'_2$, and $|T|=\nu (D_2)+1$, we see that $|\{u_1,u_2,u_3,u_4
\}\cap T|=1$.  Let $T'=T-\{u_1,u_2,u_3,u_4\}$.  If $u_1\in T$ or
$u_2\in T$, then $T'$ shows that (c) holds and if $u_3\in T$ or $u_4\in T$,
then $T'$ shows that (b) holds, as desired.  This proves that one of 
(a), (b), (c) holds for $i=2$.

It remains to show that one of (a), (b), (c) holds for $i=1$.  
Let $e_1,e_2$ be independent edges as in Lemma~\ref{4sum-le};
for the purpose of this paragraph we may take advantage of symmetry
and assume that $e_1=u_1u_2$ and $e_2=u_2u_4$.
For $j=1,2,3,4$ let $u_jv_j$ denote the edges of $Q$ with exactly one end in $V(C)$.
Let $M_1$ be the union of $M\cap E(H_1)$ and two edges of $Q$,
one with ends $v_1v_2$ and the other with ends $v_3v_4$.
Let us consider the digraph $D'_1:=D(H_1\cup
Q\backslash E(C),M_1)$.  Then $D'_1$ is isomorphic to the graph $D_1+
\{u_1a,au_2,ab,ba,u_3b,bu_4\}$.  If $D'_1$ packs, then one of (a), (b),
(c) holds:  clearly $\tau (D'_1)>\tau (D_1)$ because $D'_1$ has a circuit
disjoint from $D_1$.  If $\nu (D'_1)\ge \nu (D_1)+2$, then (a) holds;
if $\tau (D'_1)=\tau (D_1)+1$, then let $T$ be a corresponding 
\VFS.  
If $a\in T$ then $T\cap V(D_1)\cup\{u_1\}$ proves (b).
If $b\in T$ then $T\cap V(D_2)\cup\{u_3\}$ proves (c).
Thus we may assume that $D'_1$ does not pack, and so by the hypothesis of
Proposition~\ref{pr4.8} we see that $|E(H_2)|\le |E(Q)|$.  Thus $H_2$ is
a cube by Lemma~\ref{lemmaX}.  
In particular, $H_2=G_3$ and $H_1=G_1\cup G_2$.  The
definition of $H_1$ and $H_2$ implies that $\{m_1,m_2,m_3,m_4\}\not
\subseteq E(G_1)$ or $|E(G_1)|=12$.

\begin{figure}[!htb]
\epsfysize=2.2in
\epsfbox{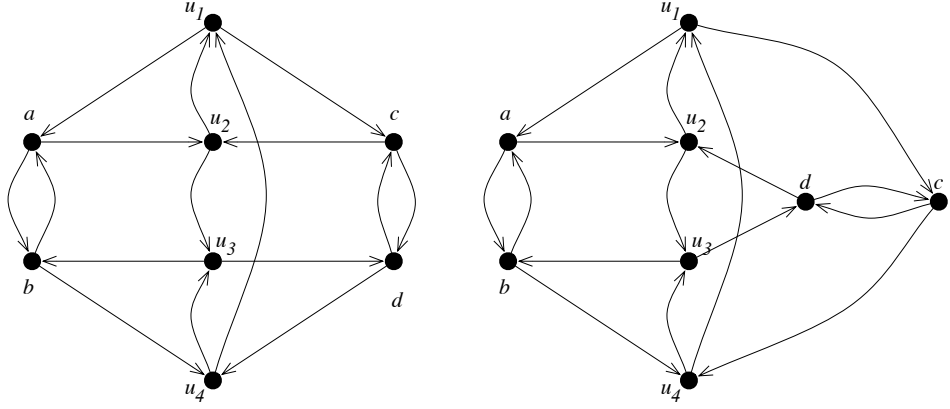}
\caption{Two digraphs.} \label{twographs}
\end{figure}

Let us first assume that $\{m_1,m_2,m_3,m_4\}\subseteq E(G_1)$.
Then $|E(G_1)|=12$, and so $G_1$ is a cube.
Since $|E(G_2)|\le |E(G_3)|$ and $G_3=H_2$ is a cube, 
we deduce that $G_1,G_2,G_3$
are all cubes.  
Let $a,b$ (resp. $c,d$) denote the edges of $M\setminus C$ in $G_2$ (resp. $G_3$).
Then $D$ is isomorphic to one of the digraphs depicted in Figure~\ref{twographs}.
For both (a) and (b), $\{u_1a,au_2,u_2u_1\},\{u_3b,bu_4,u_4u_3\},\{cd,dc\}$
is a packing of circuits and $\{a,u_3,c\}$ is a transversal.
In particular, $\nu (D)=3=\tau (D)$, a contradiction.

Thus we may assume that $\{m_1,m_2,m_3,m_4\}\not\subseteq E(G_1)$, and so
$\{m_3,m_4\}\subseteq E(G_1)$ and $\{m_1,m_2\}\subseteq E(G_2)$.
Moreover, $H_1=G_1\cup G_2$.  
For $i=1,2$ let $L_i$ be obtained 
from $G_i\backslash E(C)$
by orienting all the edges of $G_i\backslash E(C)$ from $A$ to $B$ 
and by contracting all edges of $M\cap E(G_i)$.
Then
\begin{itemize}
\item[$(*)$] $u_1$ is a source and $u_2$ is a sink of $L_1$, and 
$u_3$ is a source and $u_4$ is a sink of $L_2$.
\end{itemize}


\begin{claim}\label{clm7} \mbox{ }
\begin{enumerate}
\item 
The digraph $L_1$ does not have disjoint paths $P_1$ from $u_1$ to $u_3$ and $P_2$ from $u_4$ to $u_2$.
\item 
The digraph $L_2$ does not have disjoint paths $P_1$ from $u_3$ to $u_1$ and $P_2$ from $u_2$ to $u_4$.
\end{enumerate}
\end{claim}
\begin{cproof}
We may assume that $i=1$, and suppose for a contradiction that $P_1,P_2$ exist.
For the cube $Q$ we have $V(Q)=\{u_1,u_2,u_3,u_4,v_1,v_2,v_3,v_4\}$ and 
$E(Q)=C\cup\{u_iv_i:i=1,2,3,4\}\cup\{v_1v_2,v_2v_3,v_3v_4,v_4v_1\}$.
Let $M'=M\cup\{u_1v_1,u_2v_2,v_3v_4\}$.
Let $Q'$ be the graph obtained from $Q$ by replacing every edge of $C$ by two parallel edges.
Then $D(G_1\cup Q',M')$ contains as a subdigraph a digraph $D'$ which is obtained from $L_1$ by adding
a new vertex $w$ and edges $u_2u_1$, $u_1u_2$, $u_3w$, $wu_4$, $wu_1$, and $u_2w$.  
But that is a contradiction, because $D'$ has an odd double circuit minor (contract all but one edge of each path comprising $L_1$) and
by Theorem~\ref{maintool}, Lemma~\ref{lemmaX} and the hypothesis of Proposition~\ref{pr4.8}, $D(G_1\cup Q',M')$ packs, and hence so does $D'$.
\end{cproof}
\noindent
We now show that one of (a), (b), (c) holds for the  pair of edges $u_1u_4$ and $u_3u_2$.
Indeed, suppose that none of (a), (b), (c) hold. 
Then $D_1+u_1u_4$ has a packing of size $\nu(D_1)+1$.
This packing includes a circuit containing the edge $u_1u_4$.
Hence, $D_1$ has a packing $\zC$ of size $\nu(D_1)$ and a path $P_1$ from $u_4$ to $u_1$ disjoint from every $C\in\zC$.
Similarly, $D_1$ has a packing $\zC'$ of size $\nu(D_1)$ and a path $P_2$ from $u_2$ to $u_3$ disjoint from every $C\in\zC'$.
Since $P_1$ and $P_2$ are disjoint from any minimum \VFS\ of $D_1$ we deduce that their union is acyclic.
By $(*)$ we deduce that $P_1$ can be decomposed into either
($\alpha$) subpaths $P'_1$ from $u_2$ to $u_1$ of $L_2$ and $P''_1$ from $u_1$ to $u_3$ of $L_1$, or
($\beta$) subpaths $P'_1$ from $u_2$ to $u_4$ of $L_2$ and $P''_1$ from $u_4$ to $u_3$ of $L_1$.
Similarly, $P_2$ can be decomposed into either
($\alpha'$) subpaths $P'_2$ from $u_4$ to $u_2$ of $L_1$ and $P''_2$ from $u_2$ to $u_1$ of $L_2$, or  
($\beta'$) subpaths $P'_2$ from $u_4$ to $u_3$ of $L_1$ and $P''_2$ from $u_3$ to $u_1$ of $L_1$.
If ($\alpha$) and ($\alpha'$) occur then the paths $P''_1$ and $P'_2$ contradict Claim~\ref{clm7}(1).
If ($\beta$) and ($\beta'$) occurs then paths $P'_1$ and $P''_2$ contradict Claim~\ref{clm7}(2).
All other cases contradict the fact that $P_1\cup P_2$ is acyclic.

It remains to show that one of (a), (b), (c) holds for the  pair of edges
$u_1u_2$ and $u_3u_4$. Suppose it does not. Thus $D_1+u_3u_4$ has a packing
of size $\nu(D_1)+1$. This packing includes a circuit containing
the edge $u_3u_4$, 
and hence $D_1$ has a packing $\zC$ of size $\nu (D_1)$, and a path
$P$ from $u_4$ to $u_3$ disjoint from every member of $\zC$.
It follows from $(*)$ and Claim~\ref{clm7} that $P$ is a subgraph of $L_1$.
Since $\zC$ does not use $u_3$ or $u_4$ (because every member of
$\zC$ is disjoint from $P$) we deduce that at most one circuit of 
$\zC$ intersects both $E(L_1)$ and $E(L_2)$.  
Thus either (letting $\nu=\nu (D_1)$ and using $(*)$)
\begin{itemize}
\item[(A)] $\nu (L_1+u_3u_4) + \nu (L_2)\ge\nu+1$, or
\item[(B)] $\nu (L_1+u_2u_1+u_3u_4)+ \nu (L_2+u_1u_2)\ge \nu+2$,
\end{itemize}
where (A) (resp. (B)) occurs when no (resp. exactly one) circuit of $\zC$ intersects both $E(L_1)$ and $E(L_2)$.
Similarly, either
\begin{itemize}
\item[(C)] $\nu (L_2+u_1u_2)+\nu (L_1)\ge\nu+1$, or
\item[(D)] $\nu(L_2+u_1u_2+u_4u_3)+\nu (L_1+u_3u_4)\ge \nu+2$.
\end{itemize}

By $(*)$ $\nu (L_1)+\nu (L_2)\le \nu$.  Thus if (A) and (C) hold we deduce
that
$$\nu (D_1+u_1u_2+u_3u_4)\ge \nu (L_1+u_3u_4) + \nu (L_2+u_1u_2)=2\nu+2-
\nu (L_1)-\nu (L_2)\ge \nu+2,$$
where the first inequality follows from $(*)$.  It follows that (a) holds,
a contradiction.  Assume now that (B) and (D) hold.  Clearly
$\nu (L_2+u_1u_2+u_4u_3)\ge\nu(L_2+u_1u_2)$, $\nu(L_1+u_2u_1+u_3u_4)\ge
\nu (L_1+u_3u_4)$ and $\nu (L_1+u_2u_1+u_3u_4)+
\nu (L_2+u_1u_2+u_4u_3)\le \nu +2$.  Therefore
$$2\nu+4\ge \nu(L_1+u_2u_1+u_3u_4)+\nu(L_2+u_1u_2)+\nu(L_2+u_1u_2+u_4u_3)+
\nu (L_1+u_3u_4)\ge 2\nu+4.$$
Thus equality holds throughout, and, in particular,
$\nu(L_1+u_2u_1+u_3u_4)=\nu(L_1+u_3u_4)$.
Since $\nu (L_2)\ge \nu(L_2+u_1u_2)-1$ we have
$$\nu(L_1+u_3u_4)+\nu(L_2)\ge \nu(L_1+u_2u_1+u_3u_4)+\nu(L_2+u_1u_2)-1
\ge\nu+1$$ 
by (B),
and so (A) holds.  Thus we have shown that if (B) and (D) hold,
then (A) holds as well.  

To complete the proof we may assume that 
either (A) and (D) hold or that (B) and (C) hold.
By symmetry we may assume that the former case occurs and that (C) does not hold.
We need two claims.
\begin{itemize}
\item[(E)] $\nu(L_2+u_1u_2)\le\nu(L_2)$
\end{itemize}
To prove (E) we subtract the negation of (C) from (A), and use the fact that
$\nu(L_1+u_3u_4)\le\nu(L_1)+1$.  We find that $\nu(L_2+u_1u_2)\le\nu(L_2)$,
which is (E).
\begin{itemize}
\item[(F)] $\nu(L_2+u_4u_3)\le\nu(L_2)$
\end{itemize}
To prove (F) we use the fact that $\nu(L_1+u_3u_4)+\nu(L_2+u_4u_3)\le \nu+1$.
(Otherwise those packings could be combined to produce a packing in
$D_1$ of size $\nu+1$.)
By subtracting this inequality from (A) we obtain (F).

Let $L'_2=L_2+\{u_1a,au_2,u_3b,bu_4,ab,ba,u_4u_3\}$.  Let $Q'$ be obtained from
$Q$ by adding a three-edge path $P'$ joining $u_3$ and $u_4$, and otherwise
disjoint from $G\cup Q$.  Let $M'_2$ be a perfect matching of 
$G_2\cup Q'$ that includes $M\cap E(G_2)$, two edges of $P'$, and two edges
of $Q\backslash V(C)$: one with ends adjacent to $u_1$ and $u_2$, and
the other with ends adjacent to $u_3$ and $u_4$.  Thus $L'_2$ is 
isomorphic to $D(G_2\cup Q'\backslash E(C),M'_2)$.  The
graph $G_2\cup Q'$ is a subgraph of a brace $H$ in such a way that
$H\backslash V(G_2\cup Q')$ has a perfect matching and $H$ is obtained
from planar braces by trisumming.  By Theorem~\ref{maintool} the digraph
$L'_2$ has no minor isomorphic to an odd double circuit or
$F_7$.  By Lemma~\ref{lemmaX} the digraph $L'_2$ satisfies
$|V(L'_2)|+|E(L'_2)|<|V(D)|+|E(D)|$, and hence $L'_2$ packs by the
hypothesis of Proposition~\ref{pr4.8}.  We will show that
$\tau (L'_2)\ge\nu (L_2)+2$ and $\nu (L'_2)\le\nu (L_2)+1$.  This is 
a contradiction that will prove the proposition.  

We first show that
$\tau( L'_2)\ge\nu(L_2)+2$.  Indeed, suppose for a contradiction that
$L'_2$ has a \VFS\ $T$ of size at most $\nu(L_2)+1$.  Since $\{b,u_3,u_4\}$
is the vertex-set of a circuit of $L'_2$, one of those vertices belongs to $T$.
If $b\in T$, then $T-\{b\}$ is a \VFS\ of $L_2+u_1u_2+u_4u_3$ of size
$\nu(L_2)$.  Thus $\nu(L_2+u_1u_2+u_4u_3)+\nu(L_1+u_3u_4)\le\nu(L_2)+
\nu(L_1+u_3u_4)\le\nu+1$, contrary to (D).  If $b\not\in T$, then
$u_3\in T$ or $u_4\in T$, and $a\in T$, because $\{a,b\}$ is the vertex-set
of a circuit of $L'_2$.  Then $T-\{u_3,u_4,a\}$ is a \VFS\ of
$L_2$ by $(*)$ of size $\nu (L_2)-1$, a contradiction.  This proves
that $\tau (L'_2)\ge\nu(L_2)+2$.

Finally, it remains to prove that $\nu (L'_2)\le\nu (L_2)+1$.  To this end
suppose for a contradiction that $\zC$ is a packing in $L'_2$ of 
size $\nu (L_2)+2$.  Choose a circuit $C\in\zC$ such that $b\in V(C)$.  If
such a choice is not possible choose $C$ with $a\in V(C)$, and if that
is not possible choose $C$ arbitrarily.  It follows that the packing
$\zC-\{C\}$ uses at most one of $a$ and $u_4$, and hence the packing
$\zC-\{C\}$ proves that either $\nu(L_2+u_4u_3)>\nu(L_2)$, or
$\nu (L_2+u_1u_2)>\nu(L_2)$, contrary to (E) and (F).  This
proves that $\nu(L'_2)\le\nu(L_2)+1$, and hence completes the proof of
the proposition.  
\end{proof}
%
%
\section{Concluding remarks} \label{sec-conclusion}

Consider a digraph $D$ with weight function $w: V(D) \rightarrow Z_+$.
The weight of a subset $T \subseteq V(D)$ is defined as $\sum_{v \in T} w(v)$.
The value of the minimum weight \VFS\  is written $\tau(D,w)$.
The cardinality of the largest family $\zC$ of circuits with the property that
for every $v \in V(D)$ at most $w(v)$ circuits of $\zC$ use $v$, 
is denoted $\nu(D,w)$.
\begin{figure}[!htb]
\centerline{\epsfysize=1.8in
\epsfbox{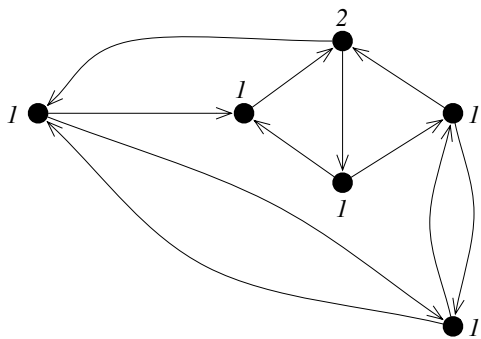}}
\caption{Digraph $D$ with $\tau(D,w) > \nu(D,w)$.} \label{w-fig}
\end{figure}
Let $e: V(D) \rightarrow Z_+$ where $e(v)=1, \forall v \in V(D)$.
Then $\tau(D)=\tau(D,e)$ and $\nu(D)=\nu(D,e)$.
Observe that for every digraph $D$ and for all positive weight functions $w$ 
we have $\tau(D,w) \geq \nu(D,w)$.
A natural extension of Theorem \ref{main-th} would be to characterize 
which are the digraphs $D$ for which $\tau(H,w)=\nu(H,w)$, 
for every subdigraph $H$ of $D$ and
for every weights $w: V(D) \rightarrow Z_+$.
This class of digraphs is closed under taking minors, and thus does not contain 
$F_7$ or odd double circuits. However, there are other obstructions as is 
illustrated by the digraph $D$ of Figure~\ref{w-fig}. 
Next to each vertex $v$ we indicate
the weight $w(v)$. Here we have $3=\tau(D,w) > \nu(D,w)=2$, and $D$ does not
contain $F_7$ or an odd double circuit as a minor. 
In fact many other obstructions can be obtained by a similar construction.
A related problem is to study the class of digraphs for which 
$\tau(D,w)=\nu(D,w)$ for 
all $w: V(D) \rightarrow Z_+$ but without requiring that the same property 
hold for every subdigraph.
This can be formulated as a hypergraph matching problem where the vertices 
of the hypergraph
are the vertices of the digraph and the edges are the vertex set of circuits 
of $D$.  There is a long list of obstructions to this property. 
However the problem has been
solved for the special case when $D$ is a tournament~\cite{Cai:98} 
or a bipartite tournament~\cite{Cai:99}.
%
%


\begin{thebibliography}{1}

\bibitem{McCuaig:97} 
W.~McCuaig,
P\'olya's permanent problem. 
{\em Electron. J. Combin.} {\bf 11} (2004), 83pp.

\bibitem{DinZan}
G.~Ding and W.~Zang, 
Packing cycles in graphs, 
{\em J.~Combin.\ Theory Ser.~B} {\bf 86} (2002), 381--407.

\bibitem{Little:75} 
C.~H.~C.~Little,
A characterization of convertible $(0,1)$-matrices,
{\em  J.~Combin.\ Theory Ser.~B} {\bf 18} (1975), 187--208.

\bibitem{Lucchesi:78}
C.~L.~Lucchesi and D.~H.~Younger,
A minimax relation for directed graphs,
{\em J. London Math. Soc. \bf17} (1978), 369--374.
 
\bibitem{ReeRobSeyTho}
B.~Reed, N.~Robertson, P.~D.~Seymour and R.~Thomas,
Packing directed circuits,
{\em Combinatorica} {\bf 16} (1996), 535--554.

\bibitem{Robertson:99}
N.~Robertson, P.~D.~Seymour and R.~Thomas, 
Permanents, Pfaffian orientations, and even directed circuits,
{\em Ann. Math.} {\bf 150} (1999), 929--975.

\bibitem{SeyTho} P.~D.~Seymour and C.~Thomassen,
Characterization of even directed graphs,
{\em J.~Combin.\ Theory Ser.~B} {\bf 42} (1987), 36-45.

\bibitem{Cai:98}W.~Zang,  M.~Cai and X.~Deng.
{\em A TDI system and its application to approximation algorithm},
{\em Proc. 39th IEEE Symposium on Foundations of Computer Science}, (1998).

\bibitem{Cai:99}
W.~Zang, M.~Cai and X.~Deng,
A min-max theorem on feedback vertex sets,
{\em Math. of Operations research},  {\bf 27}, (2002), 361-371.

\end{thebibliography}
\end{document}